\documentclass{amsart}
\usepackage{amsmath,amsthm,amssymb,amsfonts,enumerate,color,bbm}

\newcommand{\R}{\mathbb{R}}
\newcommand{\N}{\mathbb{N}}
\renewcommand{\a}{\alpha}
\renewcommand{\b}{\beta}
\newcommand{\s}{\sigma}

\newcommand{\proj}{\operatorname{Proj}}

\DeclareMathOperator*{\essup}{ess\,sup}

\newtheorem{thm}{Theorem}[section]
\newtheorem{prop}[thm]{Proposition}

\newtheorem{lem}[thm]{Lemma}

\theoremstyle{definition}

\newtheorem{rem}[thm]{Remark}

\allowdisplaybreaks

\numberwithin{equation}{section}

\author[\'O. Ciaurri and L. Roncal]{\'Oscar Ciaurri \and Luz Roncal}
\address{Departamento de Matem\'aticas y Computaci\'on\\
         Universidad de La Rioja\\
         26004 Logro\~no, Spain}
\email{oscar.ciaurri@unirioja.es, luz.roncal@unirioja.es}

\thanks{Research supported by grant MTM2012-36732-C03-02 from Spanish Government}

\keywords{Riesz transforms, harmonic oscillator, Laguerre expansions, vector-valued
inequalities, weighted inequality, mixed-norm spaces, Rubio de Francia extrapolation theorem.}

\subjclass[2010]{Primary: 42C10, 43A90, 47G40, 26A33. Secondary: 42B20, 42B35,
33C45}

\begin{document}

\title[The Riesz transform for the harmonic oscillator]
{The Riesz transform for the harmonic oscillator in spherical
coordinates}

\begin{abstract}
In this paper we show weighted estimates in mixed norm spaces for
the Riesz transform associated with the harmonic oscillator in
spherical coordinates. In order to prove the result we need a
weighted inequality for a vector-valued extension of the Riesz
transform related to the Laguerre expansions which is of
independent interest. The main tools to obtain such extension are a
weighted inequality for the Riesz transform independent of the
order of the involved Laguerre functions and an appropriate
adaptation of Rubio de Francia's extrapolation theorem.
\end{abstract}

\maketitle

\section{Introduction}
Let $H:=-\Delta +|\cdot|^2$ be the harmonic oscillator in $\R^n$.
The eigenfunctions of this operator in $\R^n$ verify
$H\phi=E\phi$, where $E$ is the corresponding eigenvalue. There
are two complete sets of eigenfunctions for $H$. By using cartesian
coordinates, one obtains the functions
\[
\phi_k(x)=\prod_{i=1}^n h_{k_i}(x_i), \qquad k=(k_1,\dots,k_n)\in
\mathbb{N}^n,
\]
where
$h_{k_i}(x_i)=(\sqrt{\pi}2^{k_i}k_i!)^{-1/2}H_{k_i}(x_i)e^{-x_i^2/2}$,
and $H_j$ denotes the Hermite polynomial of degree $j\in
\mathbb{N}$ (see \cite[p. 60]{Lebedev}). The system of functions $\{\phi_k\}_{k\in
\mathbb{N}^n }$ is orthonormal and complete in $L^2(\R^n,dx)$.

But the situation is completely different if we analyze the
eigenfunctions of the harmonic oscillator by using spherical
coordinates (see \eqref{eq:har-sphe} below). Let $\mathbb{B}^n$ be the unit
ball in $\mathbb{R}^n$ and $\mathbb{S}^{n-1}=\partial
\mathbb{B}^n$. Let $\mathcal{H}_j$ be the space of spherical
harmonics of degree $j$ in $n$ variables, and let
$\{\mathcal{Y}_{j,\ell}\}_{\ell=1,\dots,\dim{\mathcal{H}_j}}$ be
an orthonormal basis for $\mathcal{H}_j$ in
$L^2(\mathbb{S}^{n-1},d\s)$, where $\s$ is the surface area
measure in $\mathbb{S}^{n-1}$. Then the eigenfunctions of the
harmonic oscillator, see \cite{CouDeBSom}, are given by
\begin{equation}\label{eq:our system}
\tilde{\phi}_{m,k,\ell}(x)=\left(\frac{2\Gamma(k+1)}{\Gamma(m-k+n/2)}\right)^{1/2}
L_{k}^{n/2-1+m-2k}(r^2)
r^{m-2k}\mathcal{Y}_{m-2k,\ell}(x')e^{-r^2/2},
\end{equation}
where $x\in \mathbb{B}^d$, $r=|x|$, $x'\in \mathbb{S}^{n-1}$, $m\ge 0$, $k=0,\dots, [m/2]$,
$\ell=1,\dots,\dim{\mathcal{H}_{m-2k}}$, and $L_k^{b}$ are
Laguerre polynomials of order $b$ and degree $k\in \mathbb{N}$, see \cite[p. 76]{Lebedev}.
This system is orthonormal and complete in $L^2(\R^n,dx)$ and the
eigenvalues are $E_{m,k,\ell}=(n+2m)$. Moreover
\[
L^2(\R^n,dx)=\bigoplus_{m=0}^\infty \mathcal{J}_m
\]
with
\[
\mathcal{J}_m=\{f\in C^{\infty}(\R^n): Hf=(n+2m) f \}.
\]

One of the main targets of this paper will be the analysis of the Riesz
transform related to the system of eigenfunctions of the harmonic
oscillator in spherical coordinates.

It could be said that the investigation of conjugacy operators
related to discrete and continuous non-trigonometric orthogonal
expansions was initiated in the seminal article by B. Muckenhoupt
and E. M. Stein \cite{Muc-St}. They analyze a substitute of
classical conjugacy function in the context of ultraspherical
polynomials expansions, Hankel transforms and Fourier-Bessel
expansions. Later, the book by Stein \cite{St-rojo} propelled the
research in Fourier Analysis of general laplacians. It is
noteworthy to observe that in the classical one-dimensional case
the ``continuous'' counterpart of the conjugacy is the Hilbert
transform, and we have the equivalence between Hilbert transform
and the so called Riesz transform. So, abusing of the language,
the wording conjugacy and Riesz transform are used as the same
thing. For the last forty years, the research developed is huge,
and the list of references could be endless. Concerning examples
close to our context, the study of Riesz transforms in the setting
of multi-dimensional Hermite functions was initiated by S.
Thangavelu \cite{Th1, Th2, Th3} and continued in
\cite{Torrea-Stempak,Harboure-DeRosa-Segovia-Torrea,Lust}. On the
other hand, Riesz transforms associated with expansions based on
different multi-dimensional Laguerre functions have been
investigated by A. Nowak and K. Stempak in \cite{Now-Ste-JFA,
Now-Ste-Ad}.

We will analyze the Riesz transform associated with the system
given in \eqref{eq:our system} in the so called mixed norm spaces
$L^{p,2}(\R^n,r^{n-1}\,dr\,d\sigma)$ (see Section
\ref{sec:results} for definition). These spaces were first systematically studied by A. Benedek and R. Panzone in \cite{Benedek-Panzone}. They arise frequently
in harmonic analysis when the spherical harmonics are involved.
The papers \cite{Rubio,Cor,CarRomSor,BalCor} contain good examples
of their use. In \cite{CR}, the authors
considered this kind of spaces to study fractional integrals
related to the functions $\tilde{\phi}_{m,k,\ell}$.

The boundedness properties of the Riesz transform related to
$\tilde{\phi}_{m,k,\ell}$ in the mixed norm spaces will be reduced
to two inequalities due to the decomposition of the harmonic
oscillator in spherical coordinates. The first one will be a
vector-valued inequality for a sequence of Riesz transforms for
Laguerre expansions of convolution type and order $(n-2)/2+k$ (note that $\tilde{\phi}_{m,k,\ell}(x)=\ell_{k}^{n/2-1+m-2k}(r)r^{m-2k}\mathcal{Y}_{m-2k,\ell}(x')$, where $\ell_k^\alpha$ denotes the Laguerre functions which are defined in \eqref{ec:laguerre} below).
This is the main point in the proof of our result and it requires
very precise estimates of the kernel of the Riesz transform in the
Laguerre setting in terms of the order. With these estimates we
will be able to apply the Calder\'{o}n-Zygmund theory. Then a  suitable
 version of the extrapolation theorem of Rubio de Francia will do the rest
to deduce the vector-valued extension. The second inequality
appearing from the angular part
of the harmonic oscillator, will be deduced from the
Calder\'{o}n-Zygmund theory as well.

\section{The Riesz transform for $H$ in spherical coordinates}
\label{sec:results}
The harmonic oscillator in spherical coordinates can be written as
\begin{equation}
\label{eq:har-sphe} H=-\frac{\partial^2}{\partial
r^2}-\frac{n-1}{r}\frac{\partial}{\partial
r}+r^2-\frac{1}{r^2}\Delta_0,
\end{equation}
where $\Delta_0$ is the spherical part of the Laplacian. It can be
checked that
\[
H=\delta^{*}\delta +n,
\]
with
\begin{equation}
\label{eq:delta-osci}
\delta=x' \left(\frac{\partial}{\partial
r}+r\right)+\frac{1}{r}\nabla_0, \qquad \delta^{*}=-x' \left(\frac{\partial}{\partial
r}-r\right)-\frac{1}{r}\nabla_0,
\end{equation}
where $\nabla_0$ is the spherical gradient, which is the spherical part of $\nabla$ and it involves
only derivatives in $x'$. Moreover $\Delta_0=\nabla_0\cdot\nabla_0$.

For each $\s>0$, we define the fractional integrals for the harmonic
oscillator as
\[
H^{-\s}f=\sum_{m=0}^\infty
\frac{1}{(n+2m)^{\s}}\proj_{\mathcal{J}_m}f,
\]
where
\[
\proj_{\mathcal{J}_m}f=\sum_{k=0}^{\left[\frac{m}{2}\right]}
\sum_{\ell=1}^{\dim{\mathcal{H}_{m-2k}}}c_{m,k,\ell}(f)
\tilde{\phi}_{m,k,\ell}, \qquad
c_{m,k,\ell}(f)=\int_{\R^n}f(y)\overline{\tilde{\phi}_{m,k,\ell}}(y)\,
dy.
\]
With the previous definitions the Riesz transform is given as
\[
Rf:=| \delta H^{-1/2}f|.
\]
Observe that $\eqref{eq:delta-osci}$ is $\nabla+x$ in spherical coordinates, being $\nabla$ the usual gradient. Hence $Rf$ coincides with operator $|(\nabla+x)H^{-1/2}|$.

In order to analyze this kind of operators we introduce the mixed norm
spaces, defined as
\[
L^{p,2}(\R^n,r^{n-1}\,dr \, d\sigma)=\{f(x):
\|f\|_{L^{p,2}(\R^n,r^{n-1}\,dr \, d\sigma)}<\infty\},
\]
where
\[
\|f\|_{L^{p,2}(\R^n,r^{n-1}\,dr \, d\sigma)}=\Big(\int_0^\infty
\Big(\int_{\mathbb{S}^{n-1}}|f(rx')|^2\, d\sigma(x')\Big)^{p/2}\,
r^{n-1}\, dr\Big)^{1/p},
\]
with the obvious modification in the case $p=\infty$. The main
feature of these spaces is that we consider the $L^2$-norm
in the angular part and the $L^p$-norm in the radial one.
They are
very different from $L^p(\R^n,dx)$; in fact $L^p(\R^n,dx)\subset
L^{p,2}(\R^n,r^{n-1}\,dr \, d\sigma)$ for $p>2$, $L^2(\R^n,dx)=
L^{2,2}(\R^n,r^{n-1}\,dr \, d\sigma)$, and
$L^{p,2}(\R^n,r^{n-1}\,dr \, d\sigma)\subset L^{p}(\R^n,dx)$ for
$p<2$. These spaces are the most suitable when spherical harmonics
are involved due to the orthogonality of the system in the sphere $\mathbb{S}^{n-1}$. Indeed, if a function $f$ on $\R^n$ is expanded in
spherical harmonics,
\begin{equation*}
f(x)=\sum_{j=0}^\infty
\sum_{\ell=1}^{\dim{\mathcal{H}_j}}f_{j,\ell}(r)\mathcal{Y}_{j,\ell}(x'),
\end{equation*}
where
\begin{equation}
\label{ec:descom}
f_{j,\ell}(r)=\int_{\mathbb{S}^{d-1}}f(rx')\overline{\mathcal{Y}_{j,\ell}}(x')\,
d\sigma(x'),
\end{equation}
we have
\[
\|f\|_{L^{p,2}(\R^n,r^{n-1}\,dr \, d\sigma)}=
\Big\|\Big(\sum_{j=0}^\infty
\sum_{\ell=1}^{\dim{\mathcal{H}_j}}|f_{j,\ell}(r)|^2\Big)^{1/2}
\Big\|_{L^p(\R_+,r^{n-1}\, dr)}.
\]

To establish our result related to weighted inequalities
for the Riesz transform, we have to define the class of weights
involved in them. For $1\le p <\infty$, we denote by
$A_p^{\a}=A_p^{\a}(\R_+,d\mu_{\a})$ the Muckenhoupt class of $A_p$
weights on the space $(\R_+, d\mu_{\a},|\cdot|)$, where $$
d\mu_{\a}(x)=x^{2\a+1}dx.
$$
More precisely,
$A_p^{\a}$ is the class of all nonnegative functions $w\in
L_{\mathrm{loc}}^1(\R_+,d\mu_{\a})$ such that $w^{-p'/p}\in
L_{\mathrm{loc}}^1(\R_+,d\mu_{\a})$, where $1/p+1/p'=1$, and
\begin{equation*}
\sup_{I\in
\mathcal{I}}\Big(\frac{1}{\mu_{\a}(I)}\int_Iw\,d\mu_{\a}\Big)
\Big(\frac{1}{\mu_{\a}(I)}\int_Iw^{-p'/p}\,d\mu_{\a}\Big)^{p/p'}<\infty
\end{equation*}
when $1<p<\infty$, or
$$\displaystyle{\sup_{I\in
\mathcal{I}}\frac{1}{\mu_{\a}(I)}\int_Iw\,d\mu_{\a}
\essup_{x\in I}w^{-1}<\infty}$$
if $p=1$; here $\mathcal{I}$ is the class of all intervals in
$(\R_+,|\cdot|)$.

One of the main results of the paper is stated below.

\begin{thm}
\label{th:osci} Let $n\ge 2$, $1<p<\infty$, and $w\in
A_{p}^{n/2-1}$. Then
\[
\|Rf\|_{L^{p,2}(\R^n,w(r)r^{n-1}\,dr \, d\sigma)}\le
C\|f\|_{L^{p,2}(\R^n,w(r)r^{n-1}\,dr \, d\sigma)},
\]
 for each $f\in
L^{p,2}(\R^n,w(r)r^{n-1}\,dr \, d\sigma)$ and
with a constant $C$ depending on $n$ and $w$ only.
\end{thm}

The proof of Theorem \ref{th:osci} will be given in Section
\ref{sec:proof}. The main estimates will be developed in Section
\ref{sec:Riesz-Laguerre} and Section \ref{sec:angular}.

\section{Vector-valued inequalities for the Riesz transform for Laguerre expansions of convolution type}
\label{sec:Riesz-Laguerre}
Let $\a>-1$, consider the differential operator given by
\begin{equation}\label{eq:differential operator}
L_{\a}=-\frac{d^2}{dx^2}+x^2-\frac{2\a+1}{x}\frac{d}{dx},
\end{equation}
which is symmetric on $\R_+$ equipped with the measure $d\mu_\a$.
The Laguerre functions $\ell_k^{\a}$ are defined by
\begin{equation}
\label{ec:laguerre}
\ell_{k}^{\a}(x)=\Big(\frac{2\Gamma(k+1)}{\Gamma(k+\a+1)}\Big)^{1/2}
L_{k}^{\a}(x^2)e^{-x^2/2}, \quad x>0, \quad \alpha>-1.
\end{equation}
The
functions $\ell_k^{\a}$ are eigenfunctions of the differential
operator \eqref{eq:differential operator}. Indeed, we have
$L_{\a}\ell_k^{\a}=(4k+2\a+2)\ell_k^{\a}$. Furthermore, the system
$\{\ell_k^{\a}\}_{k\in\N}$ is an orthonormal basis of $L^2(\R_+,
d\mu_{\a})$. We will refer to the functions $\ell_k^\a$ as
\textit{Laguerre functions of convolution type}.

It is easily seen that $L_{\a}$ can be decomposed as
\begin{equation*}
L_{\a}=\delta^*_\a\delta_\a+2(\a+1),
\end{equation*}
where
$$
\delta_\a=\frac{d}{dx}+x,
$$
and
$$
\delta^*_\a=-\frac{d}{dx}+x-\frac{2\a+1}{x}.
$$

We provide now the definition of the Riesz transforms.
Since the spectrum of $L_{\a}$ is separated from zero, we can define the fractional integrals of order $\s$, for each $\s>0$, as
\begin{equation}
\label{eq:frac-Laguerre}
(L_{\a})^{-\s}f=\sum_{k=0}^{\infty}\frac{1}{(4k+2\a+2)^{\s}}\mathcal{P}_k^{\a}f,
\end{equation}
with $\mathcal{P}_k^{\a}f=\langle
f,\ell_k^{\a}\rangle_{d\mu_{\a}}\ell_k^{\a}$, where $\langle
f,g\rangle_{d\mu_{\a}}$ means
$\int_{\R_+}f(x)\overline{g(x)}\,d\mu_{\a}(x)$.
Now, by using $\frac{d}{dx}L_k^{\a}=-L_{k-1}^{\a+1}$, $\a>-1$, $k\in \N$, see \cite[(4.18.6)]{Lebedev}, we obtain
\begin{equation*}
\delta_{\a} \ell_k^{\a}=-2\sqrt{k}x\ell_{k-1}^{\a+1}.
\end{equation*}
Therefore, for $f\in L^2(\R_+,d\mu_{\a})$ with the expansion
$f=\sum_{k}\langle f,\ell_{k}^{\a}\rangle_{d\mu_{\a}}\ell_k^{\a}$,
we define the Riesz transform for the expansions of Laguerre functions of convolution type as
\begin{equation}\label{eq:defRiesz espectral}
\mathcal{R}^{\a}f=\delta_\a (L_\a)^{-1/2}f=-2\sum_{k=0}^{\infty}\Big(\frac{k}{4k+2\a+2}\Big)^{1/2}
\langle f,\ell_{k}^{\a}\rangle_{d\mu_{\a}}x\,\ell_{k-1}^{\a+1}.
\end{equation}
The system $\{x\,\ell_{k-1}^{\a+1}\}_{k\in \N}$ is an orthonormal
basis in $L^2(\R_+, d\mu_{\a})$, see \cite[Proposition
4.1]{Now-Ste-Ad}. Therefore, the series above converges in
$L^2(\R_+, d\mu_{\a})$ and defines a bounded operator therein.

Our result about the Riesz transform for the Laguerre expansions is the following.
\begin{thm}\label{th:LpLq Lag convolution dim1}
Let $\a\ge -1/2$, $a\ge 1$, and $1< p,r< \infty$.
Define $u_j(x)=x^{aj}$, $x\in\R_+$, $j=0,1,\ldots$.
Then there exists a constant $C$ such that
\begin{equation*}
\Big\|\Big(\sum_{j=0}^\infty|u_j\mathcal{R}^{\a+aj}(u_j^{-1}f_j)|^r\Big)^{1/r}
\Big\|_{L^p(\R_+,w\,
d\mu_{\a})} \le
C\Big\|\Big(\sum_{j=0}^\infty|f_j|^r\Big)^{1/r}\Big\|_{L^p(\R_+,
w\,d\mu_{\a})},
\end{equation*}
for all $w\in A_{p}^\a$. Moreover the constant $C$ depends on $\a$ and $w$ only.
\end{thm}

In order to prove Theorem \ref{th:LpLq Lag convolution dim1}, we
need two ingredients.

\begin{prop}\label{th:weighted inequality}
Let $\a\ge-1/2$, $a\ge 1$, and $1< r< \infty$. Define $u_j(x)=x^{aj}$,
$x\in\R_+$, $j=0,1,\ldots$. Then,
\begin{equation*}
\int_0^\infty|u_j\mathcal{R}^{\a+aj}(u_j^{-1}f)(x)|^rw(x)\,d\mu_{\a}(x)\le
C\int_0^{\infty}|f(x)|^rw(x)d\mu_{\a}(x),
\end{equation*}
for every weight $w\in A_r^\a$ and
with $C$ independent of $j$ and depending on $\a$ and $w$.
\end{prop}

\begin{prop}\label{th:GC-RFmeasure}
Let $\{T_j\}$ be a sequence of operators and suppose that,
for some fixed $r>1$, these operators are uniformly bounded in
$L^r(\R_+,w\,d\mu_{\a})$ for every weight $w\in A_r^{\a}$, i. e.
\begin{equation*}
\int_0^\infty|T_jf(x)|^rw(x)\,d\mu_{\a}(x)\le
C\int_0^{\infty}|f(x)|^rw(x)\,d\mu_{\a}(x),
\end{equation*}
with $C$ independent of $j$. Then the vector valued
inequality
\begin{equation*}
\Big\|\Big(\sum_{j=0}^\infty|T_jf_j|^r\Big)^{1/r}\Big\|_{L^p(\R_+, w\,
d\mu_{\a})} \le
C\Big\|\Big(\sum_{j=0}^\infty|f_j|^r\Big)^{1/r}\Big\|_{L^p(\R_+,w\,
d\mu_{\a})}
\end{equation*}
holds for all $1<r,p<\infty$ and $w\in A_p^{\a}$.
\end{prop}

Theorem \ref{th:LpLq Lag convolution dim1}
can be deduced easily by combining Proposition \ref{th:weighted
inequality} and Proposition \ref{th:GC-RFmeasure}. Indeed, by taking $T_j=u_j\mathcal{R}^{\a+aj}(u_j^{-1}f)$ we have that $T_j$ is bounded on $L^r(\R_+,w\,d\mu_{\a})$ for $1<r<\infty$ and all $w\in A_r^{\a}$, uniformly in $j\ge0$, by Proposition \ref{th:weighted inequality}. Then, Proposition \ref{th:GC-RFmeasure} does the rest.

In Subsection \ref{subsec:weighted inequality} we will prove
Proposition \ref{th:weighted inequality} and the proof of
Proposition \ref{th:GC-RFmeasure} will be given in Subsection
\ref{subsec:extra}.

Notation. The constants that do not depend on relevant quantities
will be denoted by $C$ and can change from one line to another
without further comment. We also note that a constant denoted by
$C_{\a}$ depends on $\a$ but not on $j$.

\subsection{Proof of Proposition \ref{th:weighted inequality}}
\label{subsec:weighted inequality}
The proof of Proposition \ref{th:weighted inequality} is based on
the theory of Calder\'on-Zygmund operators defined on spaces of
homogeneous type, where the classical weighted Calder\'on-Zygmund theory
is still valid with proper adjustments. Indeed, we can adapt the proof in the classical
case for the Lebesgue measure, see \cite{Duolibro}. In order to do this it is enough to have at our disposal the weighted  boundedness of an appropriate maximal operator. In our case we should  consider
the boundedness on $L^p(\R_+,w\,
d\mu_{\a})$, with $w\in A_p^\a$, for $1<p<\infty$ of the maximal Hardy-Littlewood operator defined as
\begin{equation}
\label{eq:max-calderon}
M_\a f(x)=\sup_{x\in I}\frac{1}{\mu_\a(I)}\int_{I}|f(y)|\,
d\mu_\a(y),
\end{equation}
which follows from a more general result by A. Calder\'on \cite{Calderon}.

We will write the
operator $u_j\mathcal{R}^{\a+aj}(u_j^{-1}f)$ as an integral
operator in the Calder\'on-Zygmund sense with kernel as follows
$$
u_j\mathcal{R}^{\a+aj}(u_j^{-1}f)(x)=\int_0^{\infty}(xy)^{aj}R^{\a+aj}(x,y)f(y)\,d\mu_\a(y),
$$
where $R^{\a+aj}$ is the kernel of the Riesz transform associated with the orthonormal system $\{\ell_{k}^{\a+aj}\}_{k\ge 0}$.
Then, we will prove growth and smoothness estimates for the
kernel. These estimates are contained in the proposition below, that is the heart of the matter.
\begin{prop}\label{prop:CZestimates} Let $\a\ge -1/2$, $a\ge 1$, and $j\ge
0$. Then
\begin{equation}\label{eq:growth}
|(xy)^{aj}R^{\a+aj}(x,y)|\le \frac{C_1}{\mu_{\a}(B(x,|x-y|))},
\quad x\neq y,
\end{equation}
\begin{equation}\label{eq:smooth}
|\nabla_{x,y}[(xy)^{aj}R^{\a+aj}(x,y)]|\le
\frac{C_2}{|x-y|\mu_{\a}(B(x,|x-y|))}, \quad x\neq y,
\end{equation}
with $C_1$ and  $C_2$ independent of $j$, and where
$\mu_{\a}(B(x,|x-y|))=\int_{B(x,|x-y|)}\,d\mu_\a$ and $B(x,|x-y|)$
is the ball of center $x$ and radius $|x-y|$.
\end{prop}
The case $j=0$  of the previous proposition is contained in
\cite{Now-Ste-Ad}, so we will focus on the proof of
\eqref{eq:growth} and \eqref{eq:smooth} for $j\ge 1$ only.

The heat semigroup
related to $L_{\a}$ is initially defined in
$L^2(\R_+,d\mu_{\a})$ as
$$
T_{\a,t}f=\sum_{k=0}^{\infty}e^{-t(4k+2\a+2)}\langle
f, \ell_k^{\a}\rangle_{d\mu_{\a}}\ell_k^{\a},\quad t>0.
$$
We can write the
heat semigroup $\{T_{\a,t}\}_{t>0}$ as an integral operator
$$
T_{\a,t}f(x)=\int_{0}^\infty G_{\a,t}(x,y)f(y)\,d\mu_{\a}(y).
$$
The Laguerre heat kernel is given by
\begin{equation*}
G_{\a,t}(x,y)=\sum_{k=0}^{\infty}e^{-t(4k+2\a+2)}
\ell_k^{\a}(x)\ell_k^{\a}(y).
\end{equation*}
The explicit expression for the Laguerre heat kernel is known and it
can be found in \cite[(4.17.6)]{Lebedev}:
$$
G_{\a,t}(x,y)=(\sinh2t)^{-1}\exp\Big(-\frac12\coth(2t)(x^2+y^2)\Big)
(xy)^{-\a}I_{\a}\Big(\frac{xy}{\sinh2t}\Big),
$$
with $I_{\a}$ denoting the modified Bessel function of the first
kind and order $\a$, see \cite[Chapter 5]{Lebedev}.

It can be seen in \cite[Section 3]{Now-Ste-Ad} that the Riesz
transform \eqref{eq:defRiesz espectral} is an operator associated,
in the Calder\'on-Zygmund sense, with the kernel
\begin{equation*}
R^{\a}(x,y)=\frac{1}{\sqrt{\pi}}\int_0^{\infty}\delta_\a G_{\a,t}(x,y)t^{-1/2}\,dt.
\end{equation*}

Let us see now that the operators $u_j\mathcal{R}^{\a+aj}(u_j^{-1}f)$
are associated, in the Calder\'on-Zygmund sense, with the kernel given by
\begin{equation}\label{eq:kernel 2}
(xy)^{aj}R^{\a+aj}(x,y)=\frac{(xy)^{aj}}{\sqrt{\pi}}\int_0^{\infty}\delta_\a G_{\a+aj,t}(x,y)t^{-1/2}\,dt.
\end{equation}
\begin{prop}\label{prop:CZ sense}
Let $\a\ge -1/2$, $a\ge 1$, and $u_j(x)=x^{aj}$, $x\in\R_+$,
$j=0,1,\ldots$. Take $f,g\in C_c^{\infty}(\R_+)$ having disjoint
supports. Let $u_j\mathcal{R}^{\a+aj}(u_j^{-1}f)$ be defined by
\eqref{eq:defRiesz espectral}. Then
$$
\langle u_j\mathcal{R}^{\a+aj}(u_j^{-1}f), g\rangle_{d\mu_{\a}}
=\int_0^{\infty}\int_0^{\infty}(xy)^{aj}R^{\a+aj}(x,y)f(y)
\overline{g(x)}\,d\mu_{\a}(y)\,d\mu_{\a}(x).
$$
\end{prop}

\begin{proof}
Let $u_j^{-1}f=h_1$ and $u_j^{-1}g=h_2$. Then,
\begin{align*}
\langle u_j\mathcal{R}^{\a+aj}(u_j^{-1}f), g\rangle_{d\mu_{\a}}&
=\langle \mathcal{R}^{\a+aj}h_1, h_2\rangle_{d\mu_{\a+aj}}\\
&=\int_0^{\infty}\int_0^{\infty}R^{\a+aj}(x,y)h_1(y)
\overline{h_2(x)}\,d\mu_{\a+aj}(y)\,d\mu_{\a+aj}(x)\\
&=\int_0^{\infty}\int_0^{\infty}(xy)^{aj}R^{\a+aj}(x,y)f(y)
\overline{g(x)}\,d\mu_{\a}(y)\,d\mu_{\a}(x),
\end{align*}
since the second identity above was proven in \cite[Proposition 3.3]{Now-Ste-Ad}.
\end{proof}
Observe that the operator $u_j\mathcal{R}^{\a+aj}(u_j^{-1}f)$ is bounded on $L^2(\R_+,d\mu_{\a})$. Indeed, its norm is the same as the norm of $\mathcal{R}^{\a+aj}$ on $L^2(\R_+,d\mu_{\a+aj})$ (this follows from the proof of Proposition \ref{prop:CZ sense}), which is proved to be finite after \eqref{eq:defRiesz espectral}.

We will find a suitable expression for the kernel \eqref{eq:kernel
2}, and this task boils down to expressing the corresponding heat
kernel in an appropriate way. We use Schl\"afli's integral
representation of Poisson's type for modified Bessel function, see
\cite[(5.10.22)]{Lebedev},
\begin{equation*}
I_{\a}(z)=z^{\a}\int_{-1}^1\exp(-zs) \,\Pi_{\a}(ds), \quad
|\arg z|<\pi, \,\,\, \a>-\frac12,
\end{equation*}
where the measure $\Pi_{\a}(du)$ is given by
\begin{equation*}
\Pi_{\a}(du)=\frac{(1-u^2)^{\a-1/2}\,du}{\sqrt{\pi}2^{\a}\Gamma(\a+1/2)}, \quad \a>-1/2.
\end{equation*}
In the limit case $\a=-1/2$, we put $\pi_{-1/2}=\tfrac12(\delta_{-1}+\delta_1)$.
Consequently, for $\a\ge -1/2$, the kernel $G_{\a,t}(x,y)$ can be expressed as
$$
G_{\a,t}(x,y)=\big(\sinh(2t)\big)^{-1-\a}\int_{-1}^1
\exp\Big(-\frac12\coth(2t)(x^2+y^2)-\frac{xys}{\sinh(2t)}\Big)\,\Pi_{\a}(ds).
$$
Let
\begin{equation*}
q_{\pm}=q_{\pm}(x,y,s)=x^2+y^2\pm2xys.
\end{equation*}
Meda's change of variable
\begin{equation*}
t=\frac12\log\frac{1+\xi}{1-\xi}, \quad \xi\in(0,1),
\end{equation*}
leads to
\begin{equation*}
G_{\a,t}(x,y)=\Big(\frac{1-\xi^2}{2\xi}\Big)^{1+\a}\int_{-1}^1
\exp\Big(-\frac{1}{4\xi}q_+(x,y,s)-\frac{\xi}{4}q_-(x,y,s)\Big)\,\Pi_{\a}(ds).
\end{equation*}

Let
\begin{equation}
\label{function-beta}
\b_{\a}(\xi)=\sqrt{\frac{2}{\pi}}\Big(\frac{1-\xi^2}{2\xi}\Big)^{1+\a}\frac{1}{1-\xi^2}
\Big(\log\Big(\frac{1+\xi}{1-\xi}\Big)\Big)^{-1/2}.
\end{equation}
In this way, by \eqref{eq:kernel 2} we get
\begin{align}
R^{\a}&(x,y)=\int_0^1\b_{\a}(\xi)\delta_{\a}\int_{-1}^1
\exp\Big(-\frac{q_+}{4\xi}-\frac{\xi q_-}{4}\Big)\,\Pi_{\a}(ds)\,d\xi \notag\\
&=\int_{-1}^1\int_0^1\b_{\a}(\xi)\Big(x-\frac{1}{2\xi}(x+ys)-\frac{\xi}{2}(x-ys)\Big)
\exp\Big(-\frac{q_+}{4\xi}-\frac{\xi q_-}{4}\Big)\,d\xi\,\Pi_{\a}(ds).\label{eq:kernel-beta}
\end{align}
The application of Fubini's theorem above can be justified, see \cite[Proposition 5.6]{Now-Ste-Ad}.

Throughout the proofs in this section, we will use several elementary facts, that are listed below.
First, observe that
\begin{equation}\label{eq:beta}
\b_{\a}(\xi)\le C 2^{-\a}\begin{cases} \xi^{-\a-3/2}, \quad 0<\xi\le 1/2,\\
\xi^{-\a-1}(1-\xi^2)^{\a}(-\log(1-\xi^2))^{-1/2}, \quad 1/2<\xi<1,
\end{cases}
\end{equation}
\begin{equation}\label{eq:des1}
 |x-ys|\le \sqrt{q_-}
 \end{equation}
 and
 \begin{equation}\label{eq:des2}
 |x-ys|^\theta\exp\Big(-\frac{\xi q_-}{4}\Big)\le C\xi^{-\theta/2}, \qquad \theta >0.
 \end{equation}
Inequality \eqref{eq:des1} is immediate, and \eqref{eq:des2} follows from \eqref{eq:des1} and the inequality
 \begin{equation}\label{eq:gammas}
 x^{\gamma}e^{-x}\le \gamma^{\gamma}e^{-\gamma}, \quad x\in \mathbb{R}_+, \quad \gamma\in \R_+.
 \end{equation}
 Let $b>0$. Define
 $h(u):=(1-u)^{b}u^{v-1/2}$, for $u\in(0,1)$. Then, for $v\ge1/2$
\begin{equation}\label{eq:funcion h}
h(u)\le \Big(\frac{b}{b+v-1/2}\Big)^{b}.
\end{equation}
We will also use frequently the following fact without further mention
$$
\frac{\Gamma(z+r)}{\Gamma(z+t)}\simeq z^{r-t}, \quad z>0, \quad r,t\in\R.
$$

Apart from this, we need several technical lemmas that provide the
tools to prove the main estimates. First we show the estimates for
the measure of the balls $B(x,(|x-y|))$ in the space $(\mathbb{R}_+,
d\mu_{\a})$. The result presented here is a one-dimensional
version of \cite[Proposition 3.2]{Now-Ste-Ad}. It will be used
tacitly throughout the proofs.
\begin{lem}\label{lem:measure}
Let $\a\ge -1/2$. Then, for all $x,y\in \mathbb{R}_+$,
$$
\mu_{\a}(B(x,|x-y|))\simeq |x-y|(x+y+|x-y|)^{2\a+1}\simeq|x-y|(x+y)^{2\a+1}.
$$
\end{lem}

We will also use the following.
\begin{lem}\label{lem:A}
Let $k, m \in \R$ be such that $m\ge -1/2$ and $k+m>-1/2$. Then
\begin{equation*}
\int_0^1\xi^{-k}\b_{m}(\xi)\exp\Big(-\frac{q_+}{4\xi}\Big)\, d\xi\le C_k \frac{2^{m}\Gamma(m+k+1/2)}{q_+^{m+k+1/2}}.
\end{equation*}
\end{lem}

\begin{proof}
Split the integral into two parts, $\int_0^{1/2}+\int_{1/2}^1$.
For the first integral, the result follows from \eqref{eq:beta} and the following estimate
$$
\int_0^{1}\xi^{-a-1}e^{-T/\xi}\,d\xi\le T^{-a}\Gamma(a), \quad a>0,
$$
which, in turn, is a slight modification of \cite[Lemma 2.1]{NoSt}.
For the second integral, from \eqref{eq:beta}, the task is reduced to estimating
$$
2^{-m}\int_{1/2}^{1}\xi^{-m-k-1}
(1-\xi^2)^{m}(-\log(1-\xi^2))^{-1/2}\exp\Big(-\frac{q_+}{4\xi}\Big)\,d\xi.
$$
Now, since $\xi\in(1/2,1)$, by \eqref{eq:gammas} and Stirling's formula we get
\begin{align*}
\xi^{-m-k-1}\exp\big(-\frac{q_+}{4\xi}\big)&\le C \xi \cdot\xi^{-m-k-1/2}\exp\big(-\frac{q_+}{4\xi}\big)\\
&\le C\xi\Big(\frac{4}{q_+}\Big)^{m+k+1/2}(m+k+1/2)^{m+k+1/2}e^{-(m+k+1/2)}\\
&\simeq C \xi\Big(\frac{4}{q_+}\Big)^{m+k+1/2}\Gamma(m+k+3/2)(m+k+1/2)^{-1/2}\\
&=C\xi\Big(\frac{4}{q_+}\Big)^{m+k+1/2}\Gamma(m+k+1/2)(m+k+1/2)^{1/2}.
\end{align*}
With this, we get
\begin{align*}
2^{-m}\int_{1/2}^{1}\xi^{-m-k-1}
&(1-\xi^2)^{m}(-\log(1-\xi^2))^{-1/2}\exp\Big(-\frac{q_+}{4\xi}\Big)\,d\xi\\
&\le C\Big(\frac{4}{q_+}\Big)^{m+k+1/2}\Gamma(m+k+1/2)(m+k+1/2)^{1/2}\\& \kern10pt \times
2^{-m}\int_{1/2}^1\xi(1-\xi^2)^{m}(-\log(1-\xi^2))^{-1/2}\,d\xi\\
&\le C_k \frac{2^{m}}{q_+^{m+k+1/2}}\,\Gamma(m+k+1/2),
\end{align*}
where in the last step we made the change of variable
$-\log(1-\xi^2)=w$ and noticed that the resulting integral is bounded
by
$\int_0^{\infty}e^{-(m+1)w}w^{-1/2}\,dw=\frac{\Gamma(1/2)}{(m+1)^{1/2}}$.
\end{proof}

The lemma below is Lemma 5.3 in \cite{CR}.
\begin{lem}\label{lem:0}
Let $c\ge -1/2$, $0<B<A$, $\lambda>0$ and $d\ge0$. Then
$$
\int_0^1\frac{(1-s)^{c+d-1/2}}{(A-Bs)^{c+d+\lambda+1/2}}\,ds\le \frac{C(d)}{A^{c+1/2}B^{d}(A-B)^{\lambda}},
$$
where
$$
C(d)=C_c\begin{cases}
\frac{\Gamma(d)\Gamma(\lambda)}{\Gamma(d+\lambda)}, & d>0,\\[5pt]
1, &d=0.
\end{cases}
$$
\end{lem}
The following Lemma will also be hepful.

\begin{lem}
\label{lem:breve}
Let $\alpha\ge -1/2$, $a\ge 1$, $j=1,2,\dots$, $k> 0$ and
\begin{equation*}
I_{k}=\int_{-1}^{1}\int_0^1 \xi^{-k}\beta_{\alpha+aj}(\xi)\exp\Big(-\frac{q_+}{4\xi}\Big)\, d\xi \,\Pi_{\alpha+aj}(ds).
\end{equation*}
Then
\begin{equation*}
I_{k}\le \frac{ C_{\alpha,k}}{(x+y)^{2\alpha+1}(xy)^{aj}|x-y|^{2k}}.
\end{equation*}
\end{lem}
\begin{proof}
By applying Lemma \ref{lem:A} with $m=\alpha+aj$ and the change of variable $s=1-2u$ we have
\[
I_k\le C_{\alpha,k} 4^{\alpha+aj}\frac{\Gamma(\alpha+aj+k+1/2)}{\Gamma(\alpha+aj+1/2)}\int_0^1 \frac{u^{\alpha+aj-1/2}(1-u)^{\alpha+aj-1/2}}{((x+y)^2-4xyu)^{\alpha+aj+k+1/2}}\, du.
\]
Then, by Lemma \ref{lem:0} with $c=\alpha$, $d=aj$ and $\lambda=k$ we conclude that
\begin{align*}
 I_k&\le
\frac{\Gamma(\alpha+aj+k+1/2)\Gamma(aj)}{\Gamma(\alpha+aj+1/2)\Gamma(aj+k)}
 \frac{ C_{\alpha,k} }{(x+y)^{2\alpha+1}(xy)^{aj}|x-y|^{2k}}\\&\le
 \frac{ C_{\alpha,k} }{(x+y)^{2\alpha+1}(xy)^{aj}|x-y|^{2k}}.\qedhere
\end{align*}
\end{proof}

Now we pass to the proof of Proposition \ref{prop:CZestimates}. Remember that we will prove \eqref{eq:growth} an \eqref{eq:smooth} for $j\ge 1$.

\subsubsection{Growth estimates: proof of \eqref{eq:growth}}

Let the kernel $R^{\a+aj}(x,y)$ be as in \eqref{eq:kernel-beta}. We write
$$
R^{\a+aj}(x,y)=J_1-\frac12J_2-\frac12J_3,
$$
where
$$
J_1:=x\int_{-1}^1\int_0^1\b_{\a+aj}(\xi)\exp\Big(-\frac{q_+}{4\xi}-\frac{\xi q_-}{4}\Big)\,d\xi\,\Pi_{\a+aj}(ds),
$$
$$
J_2:=\int_{-1}^1(x+ys)\int_0^1\xi^{-1}\b_{\a+aj}(\xi)
\exp\Big(-\frac{q_+}{4\xi}-\frac{\xi q_-}{4}\Big)\,d\xi\,\Pi_{\a+aj}(ds),
$$
and
$$
J_3:=\int_{-1}^1(x-ys)\int_0^1\xi\b_{\a+aj}(\xi)\exp\Big(-\frac{q_+}{4\xi}-\frac{\xi q_-}{4}\Big)\,d\xi\,\Pi_{\a+aj}(ds).
$$
Note that each of the three integrands of the inner integrals above are positive.

Let us begin with the study of $J_2$. First, note that
\begin{align*}
|J_2|&\le|x-y|\int_{-1}^1\int_0^1\xi^{-1}\b_{\a+aj}(\xi)
\exp\Big(-\frac{q_+}{4\xi}\Big)\,d\xi\,\Pi_{\a+aj}(ds)\\
&\,\,\,\,+y\int_{-1}^1(1+s)\int_0^1\xi^{-1}\b_{\a+aj}(\xi)
\exp\Big(-\frac{q_+}{4\xi}\Big)\,d\xi\,\Pi_{\a+aj}(ds)=:J_{21}+J_{22}.
\end{align*}
For $J_{21}$, the estimate follows from Lemma \ref{lem:breve} with $k=1$.
Concerning $J_{22}$, we consider two cases. First, the case $y>2x$
is immediate, because it can be easily seen that $J_{22}\le C
J_{21}$. For the other case, when $y\le 2x$, we use Lemma \ref{lem:A}, the change of variable $s=1-2u$ and
\eqref{eq:funcion h} with $v=\a+aj$ and $b=1/2$, to get
\begin{align*}
J_{22}&\le C\frac{\sqrt{xy}\,\Gamma(\a+aj+3/2)}{\Gamma(\a+aj+1/2)}
\int_{-1}^1\frac{(1+s)(1-s^2)^{\a+aj-1/2}}{q_+^{\a+aj+3/2}}\,ds\\
&=C\frac{\sqrt{xy}\,\Gamma(\a+aj+3/2)4^{\a+aj}}{\Gamma(\a+aj+1/2)}
\int_{0}^1\frac{u^{\a+aj-1/2}(1-u)^{\a+aj+1/2}}{((x+y)^2-4xyu)^{\a+aj+3/2}}\,du\\
&\le C\frac{\sqrt{xy}\,\Gamma(\a+aj+3/2)4^{\a+aj}}{\Gamma(\a+aj+1/2)\sqrt{\a+aj}}
\int_{0}^1\frac{(1-u)^{\a+aj}}{((x+y)^2-4xyu)^{\a+aj+3/2}}\,du\\
&\le  C_\a\frac{\sqrt{xy}\,\Gamma(\a+aj+3/2)4^{\a+aj}}{\Gamma(\a+aj+1/2)\sqrt{\a+aj}}
\frac{\Gamma(aj+1/2)}{\Gamma(aj+1)}\frac{1}{(x+y)^{2\a+1}(4xy)^{aj+1/2}|x-y|}\\
&\le \frac{C_{\a}}{(xy)^{aj}\mu_{\a}(B(x,|x-y|))}
\end{align*}
where, in the last step, we applied Lemma \ref{lem:0} with $c=\a$, $d=aj+1/2$ and $\lambda=1/2$.

We continue with $J_3$.
 It follows from \eqref{eq:des2} that
\begin{align*}
 |J_3|&\le C \int_{-1}^1\int_0^1\xi^{1/2}\b_{\a+aj}(\xi)\exp\Big(-\frac{q_+}{4\xi}\Big)\,d\xi \,\Pi_{\a+aj}(ds)\\
 &\le C\int_{-1}^1\int_0^1\xi^{-1/2}\b_{\a+aj}(\xi)\exp\Big(-\frac{q_+}{4\xi}\Big)\,d\xi \,\Pi_{\a+aj}(ds)
 \end{align*}
and the required inequality is obtained by using Lemma \ref{lem:breve} with $k=1/2$.

Finally, we study $J_1$. We split the outer integral into two
parts, $J_1=x\int_{-1}^0+x\int_0^1=:xJ_{11}+xJ_{12}$. For the
first case, observe that $x<x-ys$. This and \eqref{eq:des2} imply
\begin{align*}
x|J_{11}|&\le \int_{-1}^0|x-ys|\int_0^1\b_{\a+aj}(\xi)\exp\Big(-\frac{q_+}{4\xi}-\frac{\xi q_-}{4}\Big)
\,d\xi\,\Pi_{\a+aj}(ds)\\
&\le C \int_{-1}^1\int_0^1\xi^{-1/2}\b_{\a+aj}(\xi)\exp\Big(-\frac{q_+}{4\xi}\Big)\,d\xi
\,\Pi_{\a+aj}(ds)
\end{align*}
and we can apply again Lemma \ref{lem:breve} with $k=1/2$. Concerning $xJ_{12}$, we have $x<x+ys$ in this case. Then
$$
x|J_{12}|\le \int_{-1}^1|x+ys|\int_0^1\xi^{-1}\b_{\a+aj}(\xi)
\exp\Big(-\frac{q_+}{4\xi}\Big)\,d\xi\,\Pi_{\a+aj}(ds),
$$
so this case is reduced to that one of $J_2$, and we are done.

\subsubsection{Smoothness estimates: proof of \eqref{eq:smooth}}

Observe that
\begin{equation}\label{eq:gradiente}
\frac{d}{dx}[(xy)^{aj}R^{\a+aj}(x,y)]=(aj)x^{aj-1}y^{aj}R^{\a+aj}(x,y)
+(xy)^{aj}\frac{d}{dx}(R^{\a+aj}(x,y)).
\end{equation}
Therefore, our first aim is to get the estimate
$$
\frac{aj}{x}R^{\a+aj}(x,y)\le \frac{C}{(xy)^{aj}|x-y|\mu_{\a}(B(x,|x-y|))}.
$$
Recall from the previous subsection that $R^{\a+aj}(x,y)$ can be
written in terms of three expressions $J_1$, $J_2$ and $J_3$. We
will prove the estimate above for each one of the corresponding
expressions. The proof is systematic and follows similar reasonings as in the growth estimates, so we sketch the hints.

Let us begin with $J_2$. As in the previous subsubsection, observe that
\begin{align*}
\frac{aj}{x}|J_2|&\le C \frac{aj}{x}\frac{\Gamma(\a+aj+3/2)4^{\a+aj}}{\Gamma(\a+aj+1/2)}\Big(|x-y|
\int_0^1\frac{u^{\a+aj-1/2}(1-u)^{\a+aj-1/2}}{((x+y)^2-4xyu)^{\a+aj+3/2}}\,du\\
&\,\,\,\,+ y\int_0^1\frac{u^{\a+aj-1/2}(1-u)^{\a+aj+1/2}}{((x+y)^2-4xyu)^{\a+aj+3/2}}\,du\Big)
=:I_{21}+I_{22}.
\end{align*}
Consider now two cases. First, if $2x\ge y$. For $I_{21}$, by
\eqref{eq:funcion h} with $v=\a+aj$ and $b=1/2$, and Lemma
\ref{lem:0} with $c=\a$, $d=aj-1/2$ and $\lambda=3/2$, we have
\begin{align*}
I_{21}&\le C\frac{|x-y|}{x}\frac{\Gamma(\a+aj+3/2)4^{\a+aj}}{\Gamma(\a+aj+1/2)}\frac{aj}{\sqrt{\a+aj}}
\int_0^1\frac{(1-u)^{\a+aj-1}}{((x+y)^2-4xyu)^{\a+aj+3/2}}\,du\\
&\le C_\a \frac{|x-y|}{x}\frac{\Gamma(\a+aj+3/2)4^{\a+aj}}{\Gamma(\a+aj+1/2)}\frac{aj}{\sqrt{\a+aj}}
\frac{\Gamma(aj-1/2)}{\Gamma(aj+1)}\\
&\qquad \times
\frac{1}{(x+y)^{2\a+1}(4xy)^{aj-1/2}|x-y|^3}\\
&\le \frac{C_{\a}}{(xy)^{aj}|x-y|\mu_{\a}(B(x,|x-y|))}.
\end{align*}
Similarly, for $I_{22}$, we use \eqref{eq:funcion h} with $v=\a+aj$ and
$b=1$, and Lemma \ref{lem:0} with $c=\a$, $d=aj$ and $\lambda=1$ to get
\[
I_{22}\le \frac{C_{\a}}{(xy)^{aj}|x-y|\mu_{\a}(B(x,|x-y|))}.
\]
Secondly, if $2x<y$. For $I_{21}$, observe that $|x-y|\sim y$, use
that $u^{\a+aj-1/2}(1-u)\le C$ and apply Lemma \ref{lem:0} with
$c=\a,\, d=aj-1$ and $\lambda=2$, thus
\begin{align*}
I_{21}&\le C aj\frac{|x-y|}{x}\frac{\Gamma(\a+aj+3/2)4^{\a+aj}}{\Gamma(\a+aj+1/2)}
\int_0^1\frac{(1-u)^{\a+aj-3/2}}{((x+y)^2-4xyu)^{\a+aj+3/2}}\,du\\
&\le C aj\frac{|x-y|^2}{xy}\frac{\Gamma(\a+aj+3/2)4^{\a+aj}}{\Gamma(\a+aj+1/2)}
\frac{\Gamma(aj-1)}{\Gamma(aj+1)}\frac{1}{(x+y)^{2\a+1}(4xy)^{aj-1}|x-y|^4}\\
&\le \frac{C_{\a}}{(xy)^{aj}|x-y|\mu_{\a}(B(x,|x-y|))}.
\end{align*}
For $I_{22}$, using that $(1-u)\le C$ and $y\le 2(y-x)$, it is clear that $I_{22}\le C I_{21}$ and we are done.

Now we pass to $\tfrac{aj}{x}J_3$. Applying Lemma \ref{lem:A} and reasoning as in the previous subsubsection for $J_3$, we have
$$
\frac{aj}{x}|J_3|\le C \frac{aj}{x}\frac{ 4^{\a+aj}\Gamma(\a+aj+1)}{\Gamma(\a+aj+1/2)}\int_0^1
\frac{u^{\a+aj-1/2}(1-u)^{\a+aj-1/2}}{((x+y)^2-4xyu)^{\a+aj+1}}\,du.
$$
We distinguish two cases. If $2x\ge y$, we use \eqref{eq:funcion
h} with $v=\a+aj$ and $b=1/2$, and Lemma \ref{lem:0} with $c=\a$,
$d=aj-1/2$ and $\lambda=1$ in order to get the result. If $2x<
y$, the fact that $u^{\a+aj-1/2}(1-u)\le C$ and Lemma
\ref{lem:0} with $c=\a$, $d=aj-1$ and $\lambda=3/2$ yield the
desired estimate.

Finally, with regard to $\tfrac{aj}{x}J_1$, we split
the outer
integral into two parts,
$\tfrac{aj}{x}J_1=\tfrac{aj}{x}x\int_{-1}^0+\tfrac{aj}{x}x\int_0^{1}=:
\tfrac{aj}{x}xI_{11}+\tfrac{aj}{x}xI_{12}$. As for the first
summand, analogously to the treatment of $xJ_{11}$ in the previous
subsubsection, by \eqref{eq:des2} we have
$$
\frac{aj}{x}x|I_{11}|\le C \frac{aj}{x}\int_{-1}^1\int_0^1\xi^{-1/2}\b_{\a+aj}(\xi)
\exp\Big(-\frac{q_+}{4\xi}\Big)\,d\xi\,\Pi_{\a+aj}(ds),
$$
and this case is reduced to the study of $\tfrac{aj}{x}|J_3|$ above.
Concerning $I_{12}$, observe that, with analogous reasonings as in the previous subsubsection,
$$
\frac{aj}{x}x\,|I_{12}|\le C\frac{aj}{x}\int_{-1}^1|x+ys|\int_0^1\xi^{-1}\b_{\a+aj}(\xi)
\exp\Big(-\frac{q_+}{4\xi}\Big)\,d\xi\,\Pi_{\a+aj}(ds).
$$
and we treat the last expression as we did with $\tfrac{aj}{x}|J_2|$.

Taking into account \eqref{eq:gradiente}, we proceed now with $(xy)^{aj}\frac{d}{dx}(R^{\a+aj}(x,y))$. From the expression for the kernel in \eqref{eq:kernel-beta}, we get
\begin{multline*}
(xy)^{aj}\frac{d}{dx}(R^{\a+aj}(x,y))\\
\begin{aligned}
&=(xy)^{aj}\int_{-1}^1\int_0^1
\Big(1-\frac{1}{2\xi}-\frac{\xi}{2}\Big)\b_{\a+aj}(\xi)
\exp\Big(-\frac{q_+}{4\xi}-\frac{\xi q_-}{4}\Big)\,d\xi\,\Pi_{\a+aj}(ds)\\
&\,\,\,\,+(xy)^{aj}\int_{-1}^1\int_0^1
\Big(x-\frac{(x+ys)}{2\xi}-\frac{(x-ys)\xi}{2}\Big)\\
&\,\,\,\,\,\times\Big(-\Big(\frac{x+ys}{2}\Big)\frac{1}{\xi}-
\Big(\frac{x-ys}{2}\Big)\xi\Big)\b_{\a+aj}(\xi)
\exp\Big(-\frac{q_+}{4\xi}-\frac{\xi q_-}{4}\Big)\,d\xi\,\Pi_{\a+aj}(ds)\\
&=:(xy)^{aj}(S_1+S_2).
\end{aligned}
\end{multline*}
Concerning $S_1$, we get
\begin{equation*}
|S_1|\le C\int_{-1}^1\int_0^1
\xi^{-1}\b_{\a+aj}(\xi)
\exp\Big(-\frac{q_+}{4\xi}\Big)\,d\xi\,\Pi_{\a+aj}(ds)
\end{equation*}
and we use Lemma \ref{lem:breve} with $k=1$ to obtain the result.

The study of $S_2$ is more involved. We write
\[
S_2=\sum_{m=1}^5 S_{2m}=\sum_{m=1}^5\int_{-1}^1\int_0^1z_m(x,y,\xi)\b_{\a+aj}(\xi)
\exp\Big(-\frac{q_+}{4\xi}-\frac{\xi q_-}{4}\Big)\,d\xi\,\Pi_{\a+aj}(ds),
\]
where $z_1(x,y,\xi)=-x\Big(\frac{x+ys}{2}\Big)\frac{1}{\xi}$;
$z_2(x,y,\xi)=-x\Big(\frac{x-ys}{2}\Big)\xi$;
$z_3(x,y,\xi)=\Big(\frac{x+ys}{2\xi}\Big)^2$;
$z_4(x,y,\xi)=\frac{(x^2-y^2s^2)}{2}$; and
$z_5(x,y,\xi)=\Big(\frac{(x-ys)\xi}{2}\Big)^2$. Observe that, if
$s<0$ then $|z_1|\le\frac{|x^2-y^2s^2|}{2\xi}$; otherwise, if $s>0$ then
$|z_1|\le\Big(\frac{x+ys}{\xi}\Big)^2=|z_3|$. Note also that
$|z_4|\le C\frac{|x^2-y^2s^2|}{\xi}$. Therefore,
$|S_{21}|\le |S_{23}|+Q$, where
$$
Q:=\int_{-1}^1\int_0^1\frac{|x^2-y^2s^2|}{\xi}\b_{\a+aj}(\xi)
\exp\Big(-\frac{q_+}{4\xi}-\frac{\xi q_-}{4}\Big)\,d\xi\,\Pi_{\a+aj}(ds).
$$
In this way, if we get the desired estimates for $|S_{23}|$ and
$Q$, then we immediately obtain the same estimates for $|S_{21}|$ and
$|S_{24}|$. Concerning $Q$, by \eqref{eq:des2}, we get
\begin{align*}
Q&\le C \int_{-1}^1\int_0^1\frac{(x+ys)}{\xi^{3/2}}\b_{\a+aj}(\xi)
\exp\Big(-\frac{q_+}{4\xi}\Big)\,d\xi\,\Pi_{\a+aj}(ds)\\
& \le  C |x-y|\int_{-1}^1\int_0^1\xi^{-3/2}\b_{\a+aj}(\xi)
\exp\Big(-\frac{q_+}{4\xi}\Big)\,d\xi\,\Pi_{\a+aj}(ds)\\
&\,\,\,\, +C y\int_{-1}^1(1+s)\int_0^1\xi^{-3/2}\b_{\a+aj}(\xi)
\exp\Big(-\frac{q_+}{4\xi}\Big)\,d\xi\,\Pi_{\a+aj}(ds)
=:Q_{1}+Q_{2}.
\end{align*}
The estimate for $Q_{1}$ follows immediately from Lemma
\ref{lem:breve} with $k=3/2$. As for
$Q_{2}$, we use Lemma \ref{lem:A} with $k=3/2$ and $m=\a+aj$, \eqref{eq:funcion h} with $v=\a+aj$ and $b=1/2$, and Lemma \ref{lem:0} with $c=\a+1/2$, $d=aj$ and $\lambda=1$, so that
\begin{align*}
Q_{2}&\le C_\a (x+y)\frac{\Gamma(\a+aj+2)4^{\a+aj}}{\Gamma(\a+aj+1/2)\sqrt{\a+aj}}
\frac{\Gamma(aj)}{\Gamma(aj+1)}\frac{1}{(x+y)^{2\a+2}(4xy)^{aj}|x-y|^2}\\
&\le \frac{C_{\a}}{(xy)^{aj}|x-y|\mu_{\a}(B(x,|x-y|))}.
\end{align*}
Now, for $|S_{23}|$ we can write
\begin{align*}
|S_{23}|&\le C |x-y|^2\int_{-1}^1\int_0^1\xi^{-2}\b_{\a+aj}(\xi)\exp\Big(-\frac{q_+}{4\xi}-
\frac{\xi q_-}{4}\Big)\,d\xi\,\Pi_{\a+aj}(ds)\\
&+C y^2\int_{-1}^1(1+s)^2\int_0^1\xi^{-2}\b_{\a+aj}(\xi)\exp\Big(-\frac{q_+}{4\xi}-
\frac{\xi q_-}{4}\Big)\,d\xi\,\Pi_{\a+aj}(ds)\\&=:T_{1}+T_{2}.
\end{align*}
The
estimate for $T_{1}$ follows immediately by Lemma \ref{lem:breve} with $k=2$. Concerning $T_{2}$ we
distinguish two cases. If $y>2x$, then $y^2\simeq |x-y|^2$ and
this case is reduced to the study of $T_{1}$. If $y\le 2x$, then we use Lemma \ref{lem:A} with $k=2$ and $m=\a+aj$, \eqref{eq:funcion h} with $v=\a+aj$ and $b=1$, and
Lemma \ref{lem:0} with $c=\a$, $d=aj+1$ and $\lambda=1$, to obtain
\[
T_{2}\le \frac{C_{\a}}{(xy)^{aj}|x-y|\mu_{\a}(B(x,|x-y|))}.
\]

For $S_{22}$, observe that $|x|\le C (|x-ys|+|x+ys|$), hence
\begin{align*}
|S_{22}|&\le C \int_{-1}^1\int_0^1|x+ys|\frac{|x-ys|}{2}\xi\b_{\a+aj}(\xi)
\exp\Big(-\frac{q_+}{4\xi}-\frac{\xi q_-}{4}\Big)\,d\xi\,\Pi_{\a+aj}(ds)\\
&\,\,\,\,+C\int_{-1}^1\int_0^1|x-ys|\frac{|x-ys|}{2}\xi\b_{\a+aj}(\xi)
\exp\Big(-\frac{q_+}{4\xi}-\frac{\xi q_-}{4}\Big)\,d\xi\,\Pi_{\a+aj}(ds)\\&\le S_{24}+P,
\end{align*}
where $P$ is the second integral. The factor $S_{24}$ was already studied above, since that case
was reduced to the study of $Q$. On the other hand, we use \eqref{eq:des2} with $\theta=2$, and we get
\begin{align*}
P &\le C\int_{-1}^1\int_0^1\b_{\a+aj}(\xi)
\exp\Big(-\frac{q_+}{4\xi}\Big)\,d\xi\,\Pi_{\a+aj}(ds)\\
&\le C\int_{-1}^1\int_0^1\xi^{-1}\b_{\a+aj}(\xi)
\exp\Big(-\frac{q_+}{4\xi}\Big)\,d\xi\,\Pi_{\a+aj}(ds)
\end{align*}
and the estimate follows from Lemma \ref{lem:breve} with $k=1$. Finally, note that $S_{25}\le P$
and we are done.

\subsection{Proof of Proposition \ref{th:GC-RFmeasure}}
\label{subsec:extra}
The proof of Proposition \ref{th:GC-RFmeasure} is a
consequence of the extrapolation theorem of Rubio de Francia, see \cite{Rubio1, Rubio2},
adapted to our setting. Such adaptation is the following.
\begin{thm}\label{th:extra-RF}
Let $\alpha>-1$. Suppose that for some pair of nonnegative functions $(f,g)$, for
some fixed $1\le r<\infty$ and for all $w\in A_r^\a$ we have
$$
\int_0^\infty g(x)^rw(x)\,d\mu_{\a}(x)\le
C\int_0^{\infty}f(x)^rw(x)\,d\mu_{\a}(x),
$$
with $C$ depending only on $w$. Then, for all $1<p<\infty$ and all
$w\in A_p^\a$ we have
$$
\int_0^\infty g(x)^pw(x)\,d\mu_{\a}(x)\le C\int_0^{\infty}
f(x)^pw(x)\,d\mu_{\a}(x).
$$
\end{thm}
\begin{proof}
We follow the proof given by J. Duoandikoetxea in \cite[Theorem~3.1]{Duo}.
The main ingredients are the factorization theorem (see \cite[Lemma~2.1]{Duo})
and the construction of the Rubio de Francia weights $Rf$ and $RH$ (see \cite[Lemma~2.2]{Duo})
in our context. The first ingredient is available here because of the general factorization
theorem proved by Rubio
de Francia \cite[Section~3]{Rubio2}. For the second ingredient we use the maximal Hardy-Littlewood operator given in \eqref{eq:max-calderon}
to construct the weights $Rf$ and $RH$ as in Lemma 2.1 of \cite{Duo}.
Then the proof follows the same lines as in \cite{Duo}.
\end{proof}

\begin{rem} Originally the extrapolation theorem was given for
sublinear operators, but it turns out that sublinearity is not
necessary. Actually even the operator itself does not play any
role and all the statements can be given in terms of
pairs of nonnegative measurable functions. This observation was made by D. Cruz-Uribe and C. P\'erez in
\cite[Remark 1.11]{Cruz-Perez} and then J. Duoandikoetxea adopted this setting in
\cite{Duo}.
\end{rem}

To deduce Proposition \ref{th:GC-RFmeasure} from Theorem
\ref{th:extra-RF} we proceed in the following way. Let $\{T_j\}_{j\ge 0}$ be a sequence of operators such that each one is bounded in $L^s(\R_+,w\,d\mu_{\a})$, for some $1<s<\infty$ and for all $w\in A_s^{\a}$, uniformly in $j\ge0$. Then for $1<r<\infty$ and for any sequence of functions $f_{j}$ in $L^r(\R_+,w\,d\mu_{\a})$, we have
\begin{equation}
\label{eq:inequality con sumas}
\int_0^{\infty}\sum_{j=0}^{\infty}|T_jf_{j}(x)|^rw(x)
\,d\mu_{\a}(x)\le C \int_0^{\infty}\sum_{j=0}^{\infty}|f_{j}(x)|^rw(x)\,d\mu_{\a}(x),
\end{equation}
for all $w\in A_r^{\a}$. Now, in the extrapolation theorem above we make the following choices:
$$
g=\left(\sum_{j=0}^{\infty}|T_jf_{j}|^r\right)^{1/r},\qquad
 f=\left(\sum_{j=0}^{\infty}|f_{j}|^r\right)^{1/r}.
$$
With this pair $(f,g)$, the inequality \eqref{eq:inequality con sumas} is just the hypothesis of Theorem
\ref{th:extra-RF}. Therefore, for any $1<p<\infty$ and all $w\in A_p^{\a}$,
$$
\int_{0}^{\infty}\left(\sum_{j=0}^{\infty}|T_jf_{j}|^r\right)^{p/r}
w(x)\,d\mu_{\a}(x)\le
C\int_0^{\infty}\left(\sum_{j=0}^{\infty}|f_{j}|^r\right)^{p/r}w(x)\,d\mu_{\a}(x),
$$
and the proof is completed.

\section{Vector-valued inequalities for the fractional integrals for the Laguerre expansions of convolution type}
\label{sec:angular}
This section is devoted to the analysis of a vector-valued inequality for the fractional integrals of the Laguerre expansions of convolution type. It arises associated to the angular part, $\frac{1}{r}\nabla_0$, of the operator $\delta$ defined in \eqref{eq:delta-osci} and in the definition of the Riesz transform for the harmonic oscillator.

Let $(L_\alpha)^{-1/2}$ be the fractional integral of order $1/2$ for the Laguerre expansions as given in \eqref{eq:frac-Laguerre}.
Then, we define the operator
\[
\mathcal{T}^{\alpha}f(x)=\frac{1}{x}(L_{\alpha})^{-1/2}f(x), \quad
x\in \R_+.
\]
The boundedness properties of this operator are contained in the
following theorem.
\begin{thm}
\label{th:reminder} Let $\a\ge -1/2$, $a\ge 1$, and $1<p,r<\infty$. Define $u_j(x)=x^{aj}$, $j=1,2,\dots$. Then there
exists a constant $C$ such that
\begin{equation*}
\Big\|\Big(\sum_{j=1}^\infty|ju_j\mathcal{T}^{\a+aj}(u_j^{-1}f_j)|^r\Big)^{1/r}
\Big\|_{L^p(\R_+, w\,d\mu_{\a})} \le
C\Big\|\Big(\sum_{j=0}^\infty|f_j|^r\Big)^{1/r}\Big\|_{L^p(\R_+,
w\, d\mu_{\a})}
\end{equation*}
for all $w\in A_{p}^\a$. Moreover the constant $C$ depends only on $\alpha$ and $w$.
\end{thm}

We will deduce the result by using Proposition \ref{th:GC-RFmeasure} and the following proposition.
\begin{prop}
Let $\a\ge -1/2$, $a\ge 1$, and $1<p<\infty$. Define $u_j(x)=x^{aj}$,
$j=1,2,\dots$. Then
\begin{equation*}
\int_0^{\infty}|ju_j\mathcal{T}^{\a+aj}(u_j^{-1}f)(x)|^p w(x)\,d\mu_{\a}(x) \le
C\int_{0}^\infty |f(x)|^p w(x)\,d\mu_{\a}(x),
\end{equation*}
for every weight $w\in A_p^\alpha$, where C is a constant independent of $j$ and depending on $\a$ and $w$.
\end{prop}

The previous proposition follows from the weighted
Calder\'on-Zygmund theory. First, we need the $L^2$
estimate for the operator without weights.
\begin{prop}
Let $\a\ge -1/2$ and $a\ge 1$. Define $u_j(x)=x^{aj}$, $j=1,2,\dots$.
Then
\begin{equation}
\label{eq:angularestimate-L2} \int_{0}^\infty
|ju_j\mathcal{T}^{\a+aj}(u_j^{-1}f)(x)|^2\,d\mu_{\a}(x) \le
C\int_{0}^\infty |f(x)|^2\,d\mu_{\a}(x),
\end{equation}
for every $f\in L^2(\R_+,d\mu_\alpha)$, with $C$ independent of
$j$.
\end{prop}

\begin{proof}
Take $g=u_j^{-1}f$. Then, inequality \eqref{eq:angularestimate-L2} is implied by
\begin{equation}
\label{eq:angularestimate2-L2} \int_{0}^\infty
|(\a+aj)\mathcal{T}^{\a+aj}g(x)|^2\,d\mu_{\a+aj}(x) \le
C\int_{0}^\infty |g(x)|^2\,d\mu_{\a+aj}(x).
\end{equation}

The identity for the Laguerre polynomials \cite[p. 783,
 22.7.29]{Abra-Stegun}
\begin{equation*}
L_{k+1}^{\b+1}(x)=\frac{1}{x}((x-k-1)L_{k+1}^\b(x)+(\b+k+1)L_{k}^\b(x))
\end{equation*}
can be rewritten as
\[
\frac{\b}{x}L_k^{\b}(x)=L_{k+1}^{\b+1}(x)-L_{k+1}^{\b}(x)+\frac{k+1}{x}(L_{k+1}^{\b}(x)-L_{k}^{\b}(x)).
\]
Now we use the following \cite[p. 783,
 22.7.30]{Abra-Stegun}
\begin{equation*}
L_{k}^{\b-1}(x)=L_k^\b(x)-L_{k-1}^\b(x),
\end{equation*}
applied to the differences $L_{k+1}^{\b+1}(x)-L_{k+1}^{\b}(x)$ and $L_{k+1}^{\b}(x)-L_{k}^{\b}(x)$ and we change $x\mapsto x^2$
to deduce that
\begin{equation}
\label{ec:lag3}
\frac{\b}{x}L_k^{\b}(x^2)=\frac{k+1}{x}L_{k+1}^{\b-1}(x^2)+xL_k^{\beta+1}(x^2).
\end{equation}

Now from \eqref{ec:lag3} and the definition of the functions $\ell_k^\a$ given in \eqref{ec:laguerre}, we conclude
\[
\frac{\b}{x}
\ell_{k}^{\b}(x)=\frac{\sqrt{k+1}}{x}\ell_{k+1}^{\b-1}(x)+
x\sqrt{k+\b+1}\,\ell_{k}^{\b+1}(x).
\]
In this way, the left side of \eqref{eq:angularestimate2-L2} is
bounded by the sum of
\[
\int_0^\infty \Big|\sum_{k=0}^\infty
\sqrt{\frac{k+1}{4k+2\a+2aj+2}}\frac{\ell_{k+1}^{\a+aj-1}(x)}{x}\langle
g,\ell_{k}^{\a+aj}\rangle_{d\mu_{\a+aj}}\Big|^2\, d\mu_{\a+aj}(x),
\]
and
\[
\int_0^\infty \Big|\sum_{k=0}^\infty
\sqrt{\frac{k+\a+aj+1}{4k+2\a+2aj+2}}x\ell_{k}^{\a+aj+1}(x)\langle
g,\ell_{k}^{\a+aj}\rangle_{d\mu_{\a+aj}}\Big|^2\, d\mu_{\a+aj}(x).
\]
But the two previous summands can be controlled by the right side
of \eqref{eq:angularestimate2-L2} with a constant $C$ independent
of $j$.
\end{proof}

The operator $\mathcal{T}^{\a}f$ can be written as
\[
\mathcal{T}^{\a}f(x)=\int_0^\infty T^\a(x,y) f(y)\, d\mu_\a(y)
\]
where
\[
T^\a(x,y)=\frac{1}{x}\frac{1}{\sqrt{\pi}}\int_0^\infty
G_{\a,t}(x,y)t^{-1/2}\, dt.
\]
So, proceeding as in Proposition \ref{prop:CZ sense}, we can see
that the operators $ju_j \mathcal{T}^{\a+aj}(u_j^{-1}f)$ can be
associated, in the Calder\'on-Zygmund sense, with the kernel
\[
j(xy)^{aj}T^{\a+aj}(x,y)=\frac{j}{x}\frac{(xy)^{aj}}{\sqrt{\pi}}\int_0^\infty
G_{\a,t}(x,y)t^{-1/2}\, dt.
\]

For this kernel, the following estimates of Calder\'on-Zygmund type are
verified.

\begin{prop}
\label{prop:esti-frac}
Let $\a\ge-1/2$, $a\ge 1$, and $j\ge 1$. Then
\begin{align*}
|j(xy)^{aj}T^{\a+aj}(x,y)|&\le \frac{C_1}{\mu_\a(B(x,|x-y|))},
\quad x\not= y,\\
|j\nabla_{x,y}(xy)^{aj}T^{\a+aj}(x,y)|&\le
\frac{C_2}{|x-y|\mu_\a(B(x,|x-y|))}, \quad x\not= y,
\end{align*}
with $C_1$ and $C_2$ independent of $j$, and where
$\mu_{\a}(B(x,|x-y|))=\int_{B(x,|x-y|)}\,d\mu_\a$ and $B(x,|x-y|)$
is the ball of center $x$ and radius $|x-y|$.
\end{prop}

The proof of this proposition follows the lines of Proposition \ref{prop:CZestimates}. In this case the kernel is written as
\[
T^{\a}(x,y)=\frac{1}{x}\frac{1}{\sqrt{\pi}}\int_0^1\b_{\a}(\xi)\int_{-1}^1
\exp\Big(-\frac{q_+}{4\xi}-\frac{\xi q_-}{4}\Big)\,\Pi_{\a}(ds)\,d\xi,
\]
where $\b_\a$ is the function in \eqref{function-beta}. Then, to obtain the estimates in Proposition \ref{prop:esti-frac} we have to use Lemma \ref{lem:A}, Lemma \ref{lem:0}, and Lemma \ref{lem:breve} as it was done in the proof of Proposition \ref{prop:CZestimates}. The details are omitted.
\section{Proof of Theorem \ref{th:osci}}
\label{sec:proof}
With the change $j=m-2k$, we have
\begin{align*}
|Rf(x)|^2&=|\delta H^{-1/2}f(x)|^2=
\Big|\sum_{j=0}^\infty
\sum_{\ell=1}^{\dim \mathcal{H}_j}\sum_{k=0}^\infty
\frac{c_{2k+j,k,\ell}(f)}{\sqrt{n+2j+4k}}\delta
\tilde{\phi}_{2k+j,k,\ell}(x)\Big|^2\\&=
\Big|x'\sum_{j=0}^\infty
\sum_{\ell=1}^{\dim \mathcal{H}_j}\sum_{k=0}^\infty
\frac{c_{2k+j,k,\ell}(f)}{\sqrt{n+2j+4k}}\Big(\frac{\partial}{\partial r}+r\Big)
(r^j \ell_k^{n/2-1+j}(r))\mathcal{Y}_{j,\ell}(x')\\ &\kern40pt+\frac{1}{r}\sum_{j=0}^\infty
\sum_{\ell=1}^{\dim \mathcal{H}_j}\sum_{k=0}^\infty
\frac{c_{2k+j,k,\ell}(f)}{\sqrt{n+2j+4k}}
r^j\ell_k^{n/2-1+j}(r)\nabla_0\mathcal{Y}_{j,\ell}(x')\Big|^2\\
\\&=
\Big|x'\sum_{j=0}^\infty
\sum_{\ell=1}^{\dim \mathcal{H}_j}\sum_{k=0}^\infty
\frac{c_{2k+j,k,\ell}(f)}{\sqrt{n+2j+4k}}r^j\Big(\frac{\partial}{\partial r}+r\Big)
\ell_k^{n/2-1+j}(r)\mathcal{Y}_{j,\ell}(x')\\&\kern40pt+
x'\sum_{j=0}^\infty \frac{j}{r}
\sum_{\ell=1}^{\dim \mathcal{H}_j}\sum_{k=0}^\infty
\frac{c_{2k+j,k,\ell}(f)}{\sqrt{n+2j+4k}}
r^j \ell_k^{n/2-1+j}(r)\mathcal{Y}_{j,\ell}(x')
\Big|^2\\ &\kern40pt+\Big|\frac{1}{r}\sum_{j=0}^\infty
\sum_{\ell=1}^{\dim \mathcal{H}_j}\sum_{k=0}^\infty
\frac{c_{2k+j,k,\ell}(f)}{\sqrt{n+2j+4k}}
r^j\ell_k^{n/2-1+j}(r)\nabla_0\mathcal{Y}_{j,\ell}(x')\Big|^2,
\end{align*}
where in the last step we used that $\langle x' , \nabla_0 \mathcal{Y}_{j,\ell}(x')\rangle =0$.
Now, from the identity (see \cite[Lemma 2.2]{Perez-Pinar-Xu})
\[
\int_{\mathbb{S}^{n-1}}\langle \nabla_0 \mathcal{Y}_{j,\ell}(x')
,\nabla_0 \mathcal{Y}_{j',\ell'}(x')\rangle \, d\s
(x')=j(2j+n-2)\delta_{j,j'}\delta_{\ell,\ell'}
\]
and by using that
\[
c_{2k+j,k,\ell}(f)=\int_0^\infty\big(r^{-j} f_{j,\ell}(r)\big) \ell_{k}^{n/2-1+j}(r)r^{n-1+2j}\,dr,
\]
where $f_{j,\ell}$ is given by \eqref{ec:descom},
it can be easily checked that
\begin{align*}
\int_{\mathbb{S}^{n-1}} |Rf(x)|^2\, d\s(x')&\le \sum_{j=0}^\infty \sum_{\ell=1}^{\dim \mathcal{H}_j}
\left|r^{j}\mathcal{R}^{n/2-1+j}((\cdot)^{-j}f_{j,\ell})(r)\right|^2\\
&\,\,\,\,+C\sum_{j=1}^\infty \sum_{\ell=1}^{\dim
\mathcal{H}_j}\frac{2j+n-2}{j} \left|j
r^{j}\mathcal{T}^{n/2-1+j}((\cdot)^{-j}f_{j,\ell})(r)\right|^2.
\end{align*}
Then Theorem \ref{th:osci} is an immediate consequence of Theorem \ref{th:LpLq
Lag convolution dim1} and Theorem \ref{th:reminder} because the double sum $\sum_{j=0}^\infty \sum_{\ell=1}^{\dim \mathcal{H}_j}$ can be rearranged to become a single sum $\sum_{j=0}^\infty$.

\noindent\textbf{Acknowledgement.}  We are very thankful to the two referees of this paper for helpful and valuable remarks. We are also greatly indebted to Gustavo Garrig\'os for some fruitful discussions on the main topic of this paper.



\end{document}